\documentclass[12pt]{amsart}

\usepackage{amssymb,latexsym}

\usepackage{enumerate}

 \usepackage[french,english]{babel}

\makeatletter

\@namedef{subjclassname@2010}{

  \textup{2010} Mathematics Subject Classification}

\makeatother
\newtheorem{thm}{Theorem}[section]

\newtheorem{lem}[thm]{Lemma}
\newtheorem{pro}[thm]{Proposition}
\theoremstyle{definition}

\newtheorem{rem}[thm]{Remark}

\numberwithin{equation}{section}

\newcommand{\p}{\mathfrak{p}}
\newcommand{\X}{\mathbb{X}}

\newcommand{\ex}{\mathbb{E}}

\newcommand{\re}{\textup{Re}}
\newcommand{\im}{\textup{Im}}
\newcommand{\Psr}{\Psi_{\textup{rand}}}
\newcommand{\pr}{\mathbb{P}}

\newcommand{\M}{\mathcal{M}}
\newcommand{\N}{\mathcal{N}}
\newcommand{\Lv}{\mathbf{L}}
\newcommand{\me}{\textup{meas}}
\newcommand{\sums}{\sideset{}{^\flat}\sum}

\newcommand{\newabstract}[1]{%
  \par\bigskip
  \csname otherlanguage*\endcsname{#1}%
  \csname captions#1\endcsname
  \item[\hskip\labelsep\scshape\abstractname.]
}

\begin{document}

\baselineskip=17pt

\title[Zeros of the Epstein zeta function to the right of the critical line]{Zeros of the Epstein zeta function to the right of the critical line}

\author{Youness Lamzouri}
\address{Institut \'Elie Cartan de Lorraine, Universit\'e de Lorraine, BP 70239, 54506 Vandoeuvre-l\`es-Nancy Cedex, France; and Department of Mathematics and Statistics,
York University,
4700 Keele Street,
Toronto, ON,
M3J1P3
Canada}

\email{youness.lamzouri@univ-lorraine.fr}

\date{}

\begin{abstract} Let $E(s, Q)$ be the Epstein zeta function attached to a positive definite quadratic form of discriminant $D<0$, such that $h(D)\geq 2$, where $h(D)$ is the class number of the imaginary quadratic field $\mathbb{Q}(\sqrt{D})$. We denote by $N_E(\sigma_1, \sigma_2, T)$ the number of zeros of $E(s, Q)$ in the rectangle $\sigma_1 <\re(s)\leq \sigma_2$ and $T\leq \im(s)\leq 2T$, where $1/2<\sigma_1<\sigma_2<1$ are fixed real numbers. In this paper, we improve the asymptotic formula of Gonek and Lee for $N_E(\sigma_1, \sigma_2, T)$, obtaining a saving of a power of $\log T$ in the error term.

\end{abstract}


\thanks{The author is partially supported by a Discovery Grant from the Natural Sciences and Engineering Research Council of Canada.}

\maketitle

\section{Introduction}

The Epstein zeta functions are zeta functions associated to quadratic forms, that were introduced by Epstein \cite{Ep} in the early 1900's as generalizations of the classical Riemann zeta function. These functions are interesting analytic objects, which also have applications in algebraic number theory and the theory of modular forms. In this paper, we will only be concerned about Epstein zeta functions attached to binary quadratic forms.
Let $Q(x, y)=ax^2+bxy+cy^2$ be a positive definite quadratic form with $a, b, c\in \mathbb{Z}$, $a>0$, and discriminant $D=b^2-ac<0$. The Epstein zeta function associated to $Q$ is defined for $\re(s)>1$ by 
$$ 
E(s, Q):= \sum_{\substack{m, n \in \mathbb{Z}\\ (m,n)\neq (0, 0)}} \frac{1}{Q(m, n)^s}.
$$
It extends to a meromorphic function on $\mathbb{C}$ with a simple pole at $s=1$, and satisfies the functional equation 
\begin{equation}\label{FunctionalEq}
\left(\frac{\sqrt{-D}}{2\pi}\right)^s\Gamma(s)E(s, Q)= \left(\frac{\sqrt{-D}}{2\pi}\right)^{1-s}\Gamma(1-s)E(1-s, Q).
\end{equation}
This follows from the relation between $E(s, Q)$ and the Eisenstein series $\widetilde{E}(z, s)$, defined for $z=x+iy\in \mathbb{H}$  (where $\mathbb{H}$ is the upper-half plane) and $\re(s)>1$ by 
$$ 
\widetilde{E}(z, s):= \sum_{\substack{m, n \in \mathbb{Z}\\ (m,n)\neq (0, 0)}}\frac{y^s}{|mz+n|^{2s}}.
$$ 
Indeed, one has 
$$E(s, Q)= \left(\frac{2}{\sqrt{-D}}\right)^s \widetilde{E}(\alpha_Q, s),$$
where $\alpha_Q=(-b+\sqrt{D})/(2a)$. The functional equation \eqref{FunctionalEq} is then obtained from the analogous functional equation for $\widetilde{E}(z,s)$, which is easily derived  since the Eisenstein series $\widetilde{E}(z,s)$ is a modular form. 

Epstein zeta functions are also interesting from an arithmetic point of view, since they are related to the Dedekind zeta function $\zeta_{K}(s)$ of the imaginary quadratic field $K=\mathbb{Q}(\sqrt{D}).$ Indeed, we have 
$$ \zeta_K(s)= \frac{1}{w_D}\sums_{Q}E(s, Q),$$
where the sum $\sums_{Q}$ runs over a full set of inequivalent quadratic forms of discriminant $D$, and $w_D$ is the number of roots of unity in $K=\mathbb{Q}(\sqrt{D})$, that is
$$ w_D= \begin{cases} 6 & \text{ if } D=-3,\\ 4 & \text{ if } D=-4,\\2 & \text{ if } D<-4.
\end{cases}
$$

The distribution of zeros of $E(s, Q)$ depends on the value of the class number $h(D)$ of the imaginary quadratic field $\mathbb{Q}(\sqrt{D})$. Indeed, if $h(D)=1$ (which occurs only when $D=-3, -4, -7, -8, -11, -19, -43, -47$ and $-163$), then $E(s, Q)=w_D \zeta_K(s)$. In particular, $E(s, Q)$ has an Euler product, and is expected to satisfy an analogue of the Riemann hypothesis. However, if $h(D)\ge 2$, the distribution of zeros of $E(s, Q)$ is completely different. In this case, Davenport and Heilbronn \cite{DaHe} proved that $E(s, Q)$ has infinitely many zeros in the half-plane $\re(s)>1$. The main reason for this difference is the fact that when $h(D)\geq 2$, $E(s, Q)$ is a linear combination of two or more inequivalent $L$-functions. More precisely, one has
$$E(s, Q)=\frac{w_D}{h(D)}\sum_{\chi} \overline{\chi}(\mathfrak{a}_Q)L_K(s, \chi),$$
where  $\sum_{\chi}$ is a sum over all characters of the class group of $K=\mathbb{Q}(\sqrt{D})$,  $\mathfrak{a}_Q$ is a representative of the ideal class corresponding to the equivalence class of $Q$, and $L_K(s, \chi)$ is the Hecke $L$-function attached to $\chi$, which is defined for $\re(s)>1$ by 
$$L_K(s, \chi)= \sum_{\mathfrak{n}}\frac{\chi(\mathfrak{n})}{N(\mathfrak{n})^s}=\prod_{\mathfrak{p}}\left(1-\frac{\chi(\mathfrak{p})}{N(\mathfrak{p})^s}\right)^{-1},
$$
where $\mathfrak{n}$ and $\mathfrak{p}$ denote integer and prime ideals of $K$ respectively, and $N(\mathfrak{m})$ is the norm of the ideal $\mathfrak{m}$. 
This follows since equivalence classes
of quadratic forms of discriminant $D$ are in one-to-one correspondence with ideal classes of $K$, and the number of representations of a number $n$ by a quadratic form is the number of integer ideals of norm $n$ in the corresponding ideal class, times the number $w_D$ of roots of unity in $K$. Moreover, it is known (see for example the discussion on page 303 of \cite{GoLe}) that if $\chi$ is complex, then $L_K(s, \chi)=L_K(s, \overline{\chi}).$ Let $J$ be the number of real characters plus one half the number of complex characters of the class group of $K$, and list these characters as $\chi_1, \dots, \chi_J$ where $\chi_j\neq \chi_k$ and $\chi_j\neq \overline{\chi_k}$, for all $1\leq j\neq k\leq J$. Hence, one can write
\begin{equation}\label{EHecke}
E(s, Q)=\sum_{j=1}^J a_j L_j(s),
\end{equation}
where $L_j(s):=L_K(s, \chi_j)$ for $1\leq j\leq J$ are inequivalent Hecke $L$-functions, and 
$$
a_j:=
\begin{cases}
w_D\chi_j(\mathfrak{a}_Q)/h(D) & \text{ if } \chi_j \text{ is real},\\
2w_D\re(\chi_j(\mathfrak{a}_Q))/h(D) & \text{ if } \chi_j \text{ is complex}.\\
\end{cases}
$$

When $h(D)\geq 2$, it was conjectured by Montgomery that almost all complex zeros of $E(s, Q)$ lie on the critical line $\re(s)=1/2$. This conjecture was proved by Bombieri and Hejhal \cite{BoHe} conditionally on the Generalized Riemann Hypothesis and a weak version of a pair correlation conjecture. 

For $\sigma_1<\sigma_2$ let 
$$N_E(\sigma_1, \sigma_2, T)= \left|\left\{\rho=\beta+i\gamma, \textup{ such that } E(\rho, Q)=0,  \sigma_1<\beta \leq \sigma_2, \text{ and } T\leq \gamma \leq 2T\right\}\right|.$$
Using a universality result for Hecke $L$-functions, Voronin \cite{Vo} proved that if $h(D)\ge 2$ then for $1/2<\sigma_1<\sigma_2<1$ fixed, we have 
\begin{equation}\label{Voronin}
 N_E(\sigma_1, \sigma_2, T)\gg T,
\end{equation}
where the implicit constant depends on $\sigma_1$ and $\sigma_2$. Lee \cite{Le1} improved this result to an asymptotic formula
\begin{equation}\label{AsympEpstein}
N_E(\sigma_1, \sigma_2, T)\sim c_E(\sigma_1, \sigma_2) T,
\end{equation}
where $c_E(\sigma_1, \sigma_2)> 0$ for  $1/2<\sigma_1<\sigma_2<1$. More recently, building on the work of Lamzouri, Lester and Radziwill \cite{LLR} for the distribution of $a$-points of the Riemann zeta function, Gonek and Lee \cite{GoLe} obtained a non-trivial upper bound for the error term in  \eqref{AsympEpstein}. More precisely, they showed that if $h(D)\ge 2$ and $1/2<\sigma_1<\sigma_2<1$ are fixed, then we have 
\begin{equation}\label{GonekLee}
N_E(\sigma_1, \sigma_2, T)= c_E(\sigma_1, \sigma_2) T+ O\left(T\exp\left(-b\sqrt{\log\log T}\right)\right),
\end{equation}
for some absolute constant $b$. Using the same method, Lee \cite{Le2} improved this asymptotic formula, obtaining a saving of a power of $\log T$ in the error term, in the special case where $E(s, Q)$ is a linear combination of exactly two inequivalent $L$-functions, which corresponds to $h(D)=2$ or $h(D)=3$. More precisely, he showed that in this case 
$$ 
N_E(\sigma_1, \sigma_2, T)= c_E(\sigma_1, \sigma_2) T + O\left(T\frac{\log\log T}{(\log T)^{\sigma_1/2}}\right).
$$
However, when $h(D)>3$, $E(s, Q)$ is a linear combination of three or more inequivalent $L$-functions, and in this case, the method of Gonek and Lee only yields the weaker error term $O(T\exp(-b\sqrt{\log\log T}))$. 

In this note, we use a different and more streamlined method to improve the error term in the asymptotic formula \eqref{GonekLee}. Our approach relies on a geometric box covering argument in $\mathbb{R}^{2J}$, and gives a saving of a power of $\log T$ in the error term of \eqref{GonekLee} when $h(D)>3$.
\begin{thm}\label{Main}
Let $Q(x, y)=ax^2+bxy+cy^2$ be a positive definite quadratic form with $a, b, c\in \mathbb{Z}$, $a>0$,  and discriminant $D=b^2-ac<0$, such that $h(D)\ge 2$.  Let $1/2<\sigma_1<\sigma_2<1$ be fixed. Then, we have 
\begin{equation}\label{MainEst}
N_E(\sigma_1, \sigma_2, T)= c_E(\sigma_1, \sigma_2) T +O\left(\frac{T}{(\log T)^{\alpha+o(1)}}\right),
\end{equation}
where $\alpha=\sigma_1/(4J+2).$
\end{thm}
\begin{rem} The proof of Theorem \ref{Main} gives a quantitative estimate for the term $(\log T)^{o(1)}$ in the RHS of \eqref{MainEst}. More precisely, it follows that the error term in the asymptotic formula \eqref{MainEst} is $\ll T \exp(b\sqrt{\log\log T})/(\log T)^{\alpha}$, for some constant $b=b(\sigma_1, \sigma_2)>0$.
\end{rem}

\section{Strategy of proof of Theorem \ref{Main} and key ingredients}

Let $1/2<\sigma_1<\sigma_2<1$ be fixed real numbers, and $T$ be large. To count the number of zeros of $E(s, Q)$ in the rectangle $\sigma_1<\re(s)\leq \sigma_2$, $T\leq \im(s)\leq 2T$ we shall use Littlewood's lemma in a standard way. Let $\rho_Q= \beta_Q+ i \gamma_Q$ denote a zero of $E(s, Q)$. It is known that there exists  $\sigma_0$ such that $\beta_Q<\sigma_0$ for all zeros $\rho_Q$ of $E(s, Q)$. By Littlewood's lemma
(see equation (9.9.1) of Titchmarsh \cite{Ti}), we have 
\begin{equation}\label{Littlewood}
\begin{aligned}
\int_{\sigma}^{\sigma_0}\bigg(\sum_{\substack{ \beta_Q >u \\ T \leq \gamma_Q \leq 2T}} 1\bigg) du
& = \frac{1}{2\pi} \int_{T}^{2T} \log |E(\sigma+it, Q)| dt - \frac{1}{2\pi} \int_{T}^{2T} \log |E(\sigma_0 + it, Q)| dt\\
& \ \ +O_Q(\log T).
\end{aligned}
\end{equation}
In order to estimate the integrals on the right hand side of this asymptotic formula, we shall construct a probabilistic random model for $E(\sigma+it, Q)$. This was also used in  \cite{GoLe}, \cite{Le1} and \cite{Le2}. Recall from \eqref{EHecke} that 
$$ E(\sigma+it, Q)= \sum_{j=1}^J a_j L_j(\sigma+it).$$ 
Let $\{\X(p)\}_{p}$ be a sequence of independent random variables, indexed by the prime numbers, and uniformly distributed on the unit circle. For $1\leq j\leq J$ we consider the random Euler products 
$$
L_j(\sigma, \X):=\prod_{\mathfrak{p}} \left(1-\frac{\chi_j(\mathfrak{p})\X(p)}{N(\mathfrak{p})^{\sigma}}\right)^{-1},
$$
where $p$ is the unique rational prime dividing $N(\mathfrak{p})$. These random products converge almost surely for $\sigma>1/2$ by Kolmogorov's three series Theorem. We shall prove that $\frac{1}{T} \int_{T}^{2T}\log |E(\sigma+it, Q)|dt$ is very close to the expectation (which we shall denote throughout by $\ex(\cdot)$) of $\log |E(\sigma, \X)|$, where the probabilistic random model $E(\sigma, \X)$ is defined by 
$$ E(\sigma, \X):= \sum_{j=1}^J a_j L_j(\sigma, \X).$$
\begin{thm}\label{AsympLogE} Let $\sigma>1/2$ be fixed. There exists a constant $b=b(\sigma)>0$ such that 
$$ 
\frac{1}{T} \int_{T}^{2T}\log |E(\sigma+it, Q)|dt = \ex\left(\log |E(\sigma, \X)|\right)+ O\left(\frac{e^{b\sqrt{\log\log T}}}{(\log T)^{\sigma/(2J+1)}}\right).
$$
\end{thm}

Gonek and Lee \cite{GoLe} obtained such an asymptotic formula, but with the weaker error term $O(\exp(-b\sqrt{\log\log T}))$. 

We now show how to deduce Theorem \ref{Main} from Theorem \ref{AsympLogE} and \eqref{Littlewood}. The proof also provides an explicit description of the constant $c_E(\sigma_1, \sigma_2)$ in terms of the probabilistic random model $E(\sigma, \X)$. 
\begin{proof}[Proof of Theorem \ref{Main}]
Let 
$$ \M(\sigma)=\ex\left(\log |E(\sigma, \X)|\right). $$
Lee \cite{Le1} proved that $\M(\sigma)$ is twice differentiable as a function of $\sigma$. Let $h>0$ be small. Combining Theorem \ref{AsympLogE} with the estimate \eqref{Littlewood} at $\sigma$ and $\sigma+h$, we obtain
$$
\int_{\sigma}^{\sigma+h}\bigg(\sum_{\substack{ \beta_Q >u \\ T \leq \gamma_Q \leq 2T}} 1\bigg)  du =
\frac{T}{2\pi} (\M(\sigma)-\M(\sigma+h))+ O \left(\frac{Te^{b\sqrt{\log\log T}}}{(\log T)^{\sigma/(2J+1)}}\right).
$$
Dividing by $h$ both sides, and using that $\M(\sigma)$ is twice differentiable gives 
\begin{align*}
\frac{1}{h} \int_{\sigma}^{\sigma + h} \bigg(\sum_{\substack{ \beta_Q >u \\ T \leq \gamma_Q \leq 2T}} 1\bigg) du & = \frac{T}{2\pi} \cdot \bigg ( \frac{\M(\sigma) - \M(\sigma + h)}{h} \bigg ) + O \bigg (\frac{Te^{b\sqrt{\log\log T}}}{h(\log T)^{\sigma/(2J+1)}} \bigg ) \\
& = - \frac{T}{2\pi} \cdot \M'(\sigma) + O \bigg ( h T + \frac{Te^{b\sqrt{\log\log T}}}{h(\log T)^{\sigma/(2J+1)}} \bigg ).
\end{align*}
Therefore,
$$
\sum_{\substack{\beta_Q > \sigma + h \\ T \leq \gamma_Q \leq 2T}} 1 \leq
- \frac{T}{2\pi} \cdot \M'(\sigma) +  O \bigg ( h T + \frac{Te^{b\sqrt{\log\log T}}}{h(\log T)^{\sigma/(2J+1)}}  \bigg )
\leq \sum_{\substack{\beta_Q > \sigma \\ T \leq \gamma_Q \leq 2T}} 1.
$$
We substitute $\sigma - h$ for $\sigma$, and use that $\M'(\sigma - h) = \M'(\sigma) + O(h)$ (since $\M'(\sigma)$ is differentiable) to get
$$
\sum_{\substack{\beta_Q >\sigma \\ T \leq \gamma_{a} \leq 2T}} 1
\leq - \frac{ T}{2\pi} \cdot \M'(\sigma) + O\left(hT + \frac{Te^{b\sqrt{\log\log T}}}{h(\log T)^{\sigma/(2J+1)}}\right).
$$
We pick $h = (\log T)^{-\sigma/(4J+2)}$ to conclude that
$$
\sum_{\substack{\beta_Q > \sigma \\ T \leq \gamma_Q \leq 2T}} 1 = 
- \frac{T}{2\pi} \cdot \M'(\sigma) + O \bigg(\frac{Te^{b\sqrt{\log\log T}}}{(\log T)^{\sigma/(4J+2)}} \bigg ).
$$
Thus, using this estimate with $\sigma=\sigma_1$ and $\sigma=\sigma_2$ gives
$$ N_E(\sigma_1, \sigma_2, T)= c_E(\sigma_1, \sigma_2) T +O \bigg(\frac{Te^{b\sqrt{\log\log T}}}{(\log T)^{\sigma_1/(4J+2)}} \bigg ),$$ 
where $$c_E(\sigma_1, \sigma_2)=\frac{\M'(\sigma_2)-\M'(\sigma_1)}{2\pi}.$$

\end{proof}

Our proof of Theorem \ref{AsympLogE} (which will be given in the next section) uses a different approach, but relies on the same key ingredients as in \cite{GoLe}. The first is a discrepancy bound for the joint distribution of the Hecke $L$-functions $L_j(s)$. For $\sigma>1/2$ let $$ \Lv(\sigma+it)=\Big(\log |L_1(\sigma+it)|, \dots, \log |L_J(\sigma+it)|, \arg L_1(\sigma+it), \dots, \arg L_J(\sigma+it) \Big),$$
and similarly define the random vector
$$ \Lv(\sigma, \X)=\Big(\log |L_1(\sigma, \X)|, \dots, \log |L_J(\sigma, \X)|, \arg L_1(\sigma, \X), \dots, \arg L_J(\sigma, \X) \Big).$$
Then we have the following result, which is essentially proved by Gonek and Lee \cite{GoLe}, and is a generalization of Theorem 1.1 of \cite{LLR}. Its proof is a slight modification of  the proof of Theorem 1.2 of \cite{GoLe}, so we omit it. Here and throughout we let ``$\text{meas}$'' denotes the Lebesgue measure on $\mathbb{R}$.
\begin{thm}\label{Discrepancy}
Let $1/2<\sigma<1$ be fixed. The we have 
$$ \sup_{\mathcal{B}} \left| \frac{1}{T} \me \big \{ t\in [T, 2T]: \Lv(\sigma+it) \in \mathcal B \big\}-\pr\left(\Lv(\sigma, \X)\in \mathcal{B}\right)\right|\ll \frac{1}{(\log T)^{\sigma}},$$
where the supremum is taken over all rectangular boxes (possibly unbounded)  $ \mathcal{B}\subset \mathbb{R}^{2J}$, with sides parallel to the coordinate axes. 
\end{thm}
We shall use this result to approximate the integral $\frac 1T\int_T^{2T} \log |E(\sigma+it, Q)|dt$ by the expectation $\ex(\log|E(\sigma, \X)|)$. However, in doing so we need to control the large values and the logarithmic singularities of both $\log|E(\sigma+it, Q)|$ and $\log |E(\sigma, \X)|$. To this end we use the following lemmas, which are proved in \cite{GoLe}.

\begin{lem}[Lemma 3.1 of \cite{GoLe}]\label{MomentsLogE} Let $1/2<\sigma\leq 2$ be fixed. There exists a constant $C_1>0$ depending at most on $J$, such that for every positive integer $k$ we have 
$$ \frac 1T \int_T^{2T} \big|\log|E(\sigma+it, Q)|\big|^{2k} dt \ll (C_1k)^{4k}.$$
\end{lem}

\begin{lem}[Lemma 3.2 of \cite{GoLe}]\label{MomentsOneL}
Let $1/2<\sigma\leq 2$ be fixed, and $1\leq j\leq J$. There exist positive constants $C_2, C_3$ depending on $\sigma$, such that for every positive integer $k\leq (\log T)/(C_3 \log\log T)$ we have 
$$\frac{1}{T}\int_T^{2T}|\log L_j(\sigma+it)|^{2k}dt \ll (C_2k)^{k}.$$
\end{lem}

\begin{lem}[Lemma 3.3 of \cite{GoLe}]\label{MomentsLogR}
Let $1/2<\sigma\leq 2$ be fixed. There exists a constant $C_4>0$ depending at most on $J$, such that for every positive integer $k$ we have 
\begin{equation}\label{MomentsLogR2}\ex\left(\big| \log \left|E(\sigma, \X)\right|\big |^{2k}\right)\ll (C_4k)^{2k},
\end{equation} 
and for all $1\leq j\leq J$ 
\begin{equation}\label{MomentsOneR}\ex\left(\left| \log L_j(\sigma, \X)\right|^{2k}\right)\ll (C_4k)^{k}.
\end{equation} 
\end{lem}


\section{Proof of Theorem \ref{AsympLogE}}

We start by showing how to use Lemmas \ref{MomentsLogE} and \ref{MomentsOneL} to control the large values and the logarithmic singularities of $\log|E(\sigma+it, Q)|$. Let $A$ be a suitably large constant and  put $M=A\sqrt{\log\log T}$. We consider the following sets
$$ S_1(T):= \left\{ t\in [T, 2T] : \Lv(\sigma+it) \in (-M, M)^{2J}\right\},$$ 
$$S_2(T):= \left\{ t\in [T, 2T] : \log |E(\sigma+it, Q)|>-M^4\right\}, \text {and } S(T)= S_1(T) \cap S_2(T).$$
Let $k= \lfloor2A\log\log T\rfloor$. Then, it follows from Lemma \ref{MomentsOneL} that
\begin{equation}\label{MeasureS1}
\begin{aligned}
\me([T, 2T]\setminus S_1(T)) & \leq \sum_{j=1}^{J} \me \left\{ t\in [T, 2T] : |\log L_j(\sigma+it)| \ge M \right\}\\
& \leq \sum_{j=1}^{J} \frac{1}{M^{2k}} \int_T^{2T} |\log L_j(\sigma+it)|^{2k} dt\\
& \ll T \left(\frac{C_2 k}{M^2}\right)^{k} \ll \frac{T}{(\log T)^{2A}},
\end{aligned}
\end{equation}
if $A$ is suitably large. 
On the other hand, using Lemma \ref{MomentsLogE} with the same choice of $k$ gives\begin{align*}
 \me([T, 2T]\setminus S_2(T)) &\leq 
\frac{1}{M^{8k}} \int_{T}^{2T} \big|\log |E(\sigma+it, Q)|\big|^{2k} dt \\
&\ll T \left(\frac{C_1k}{M^2}\right)^{4k}\ll \frac{T}{(\log T)^{2A}}.\end{align*}
Therefore we deduce
$$  
\me([T, 2T]\setminus S(T))  \ll \frac{T}{(\log T)^{2A}}.
$$
Combining this bound with Lemma \ref{MomentsLogE}, and using  H\"older's inequality with $r= \lfloor \log\log T\rfloor$ we get
\begin{equation}\label{TrunIntLog}
\begin{aligned}
& \int_{t\in [T, 2T]\setminus S(T)} \log |E(\sigma+it, Q)|dt \\
& \leq \big(\me\{t\in [T, 2T]\setminus S(T)\}\big)^{1-1/2r}\left(\int_{T}^{2T}\big|\log |E(\sigma+it, Q)|\big|^{2r}dt\right)^{1/2r}\\
& \ll \left(\frac{T}{(\log T)^{2A}}\right)^{1-1/2r} \left(T(C_1k)^{4r}\right)^{1/2r} \\
& \ll \frac{T}{(\log T)^{A}}.
\end{aligned}
\end{equation}
We now define for $\tau\in \mathbb{R}$
$$ \Psi_T(\tau):= \frac{1}{T} \me\left\{ T\in S(T): \log |E(\sigma+it, Q)|>\tau\right\},$$
and similarly
$$ \Psr(\tau):= \pr( \X\in \mathcal{S}, \text { and }\log |E(\sigma, \X)|>\tau),$$
where $\mathcal{S}$ is the event $\Lv(\sigma, \X)\in (-M, M)^{2J}$ and $\log|E(\sigma, \X)|>-M^4$. 
We shall deduce Theorem \ref{AsympLogE} from the following result which shows that $\Psi_T(\tau)$ is very close to $\Psr(\tau)$ uniformly in $\tau$.
\begin{pro}\label{CoveringResult}
For $T$ large, we have
$$ \sup_{\tau\in \mathbb{R}} \left| \Psi_T(\tau)- \Psr(\tau)\right| \ll \frac{e^{2M}}{(\log T)^{\sigma/(2J+1)}}.$$
\end{pro}
\begin{proof} 
We let $0<\varepsilon= \varepsilon(T)\leq e^{-10M}$, be a small parameter to be chosen later. 
We shall consider three cases depending on the size of $\tau$. 

\smallskip

\noindent \textbf{Case 1}: $\tau\leq  -M^4.$ In this case, it follows from the definitions of the set $S(T)$ and the event $\mathcal{S}$ that
$$
  \Psi_T(\tau)= \frac{1}{T} \me\big\{ t\in [T, 2T]: \Lv(\sigma+it) \in (-M, M)^{2J}\big\},
 $$ 
 and 
 $$ \Psr(\tau)=  \pr\big(\Lv(\sigma, \X) \in (-M, M)^{2J}\big), $$
 and hence the desired estimate follows from Theorem \ref{Discrepancy}.
 
 \smallskip
 
\noindent \textbf{Case 2}: $-M^4< \tau\leq \log(2\varepsilon)+ 2M.$

In this case we have 
\begin{equation}\label{SecondCase}
  \Psi_T(\tau)= \frac{1}{T} \me\big\{ t\in [T, 2T]: \Lv(\sigma+it) \in (-M, M)^{2J}\setminus \mathcal{U}_J(e^{\tau}, M)\big\},
\end{equation}
 where $\mathcal{U}_J(y, M)$ is the bounded subset of $\mathbb{R}^{2J}$ defined by
\begin{align*}
\mathcal{U}_J(y, M)
:= \Big\{ (u_1, \dots, u_J, v_1, \dots, v_J) \in \mathbb{R}^{2J} : \  & |u_j|, |v_j| <M \text{ for all } 1\le j\le J,\\
& \text{ and } \big|\sum_{j=1}^J a_j e^{u_j+iv_j}\big|\leq y \Big\}.
\end{align*}
 We cover $\mathcal{U}_J(e^{\tau}, M)$ with $K$ hypercubes $\mathcal{B}_k(\tau)$ (of dimension $2J$)
 with non-empty intersection with $\mathcal{U}_J(e^{\tau}, M)$, and with
sides of length $\varepsilon$. The number of such hypercubes is 
$$ K \asymp \frac{\textup{Vol}(\mathcal U_J(e^{\tau}, M))}{\varepsilon^{2J}} \ll \left(\frac{M}{\varepsilon}\right)^{2J}.
$$
Now, let $1\leq k\leq K$ and $(u_1, \dots, u_J, v_1, \dots, v_J) \in \mathcal{B}_k(\tau) \cap \mathcal{U}_J(e^{\tau}, M)$ (recall that this intersection is non-empty by construction). Then, for any $(x_1, \dots, x_J, y_1, \dots, y_J)\in \mathcal{B}_k(\tau)$ we have $|x_j-u_j|\leq \varepsilon$ and 
$|y_j-v_j|\leq \varepsilon$ for all $1\leq j\leq J$. Hence, we deduce that $|x_j|, |y_j|<2M$ for all $1\leq j\leq J$ and 
$$ \big|\sum_{j=1}^J a_j e^{x_j+iy_j}\big|=\big|\sum_{j=1}^J a_j e^{u_j+iv_j}\big| +O(\varepsilon e^M) \leq C_5\varepsilon e^{2M}$$
for some positive constant $C_5$ since $e^{\tau} \leq 2\varepsilon e^{2M}$ by our assumption. Therefore, we have shown that
\begin{equation}\label{UpperLowerBoxes}
\mathcal{U}_J(e^{\tau}, M) \subset \left(\bigcup_{k \le K} \mathcal{B}_k(\tau) \right) \subset \mathcal{U}_J(C_5\varepsilon e^{2M}, 2M).
\end{equation}
By Theorem \ref{Discrepancy} we thus deduce that 
\begin{equation}\label{BoundMesExceptional}
\begin{aligned}
& \frac{1}{T} \me\big\{ t\in [T, 2T]: \Lv(\sigma+it) \in \mathcal{U}_J(e^{\tau}, M)\big\}\\
& \leq 
\sum_{k=1}^K \frac{1}{T} \me\big\{ t\in [T, 2T]: \Lv(\sigma+it) \in \mathcal{B}_k(\tau)\big\}\\
&= \sum_{k=1}^K\pr\left(\Lv(\sigma, \X)\in \mathcal{B}_k(\tau)\right) + O\left(\frac{K}{(\log T)^{\sigma}}\right)\\
&\leq \pr\left(\Lv(\sigma, \X)\in \mathcal{U}_J(C_5\varepsilon e^{2M}, 2M)\right) + O\left(\frac{(\log\log T)^J}{\varepsilon^{2J}(\log T)^{\sigma}}\right).
\end{aligned}
\end{equation}
Moreover, it follows from the work of Borchsenius and Jessen \cite{BoJe} (see for example page 315 of \cite{GoLe}) that $E(\sigma, \X)$ is an absolutely continuous random variable. This shows that 
\begin{equation}\label{BoundPrExceptional}
\pr\left(\Lv(\sigma, \X)\in \mathcal{U}_J(C_5\varepsilon e^{2M}, 2M)\right) \leq \pr\left(|E(\sigma, \X)| \leq C_5\varepsilon e^{2M}\right)\ll \varepsilon e^{2M}.
\end{equation}
Combining this bound with \eqref{SecondCase} and \eqref{BoundMesExceptional} gives
$$ \Psi_T(\tau)= \frac{1}{T} \me\big\{ t\in [T, 2T]: \Lv(\sigma+it) \in (-M, M)^{2J}\big\} + O\left(\varepsilon e^{2M}+ \frac{(\log\log T)^J}{\varepsilon^{2J}(\log T)^{\sigma}}\right).$$
Similarly, it follows from \eqref{BoundPrExceptional} that
$$ \Psr(\tau)= \pr\left(\Lv(\sigma, \X) \in (-M, M)^{2J}\right) +O(\varepsilon e^{2M}).$$
The desired bound on the discrepancy $|\Psi_T(\tau)- \Psr(\tau)|$ then follows from Theorem \ref{Discrepancy} by choosing $\varepsilon= (\log T)^{-\sigma/(2J+1)}$.

\smallskip

\noindent \textbf{Case 3}: $\tau> \log(2\varepsilon)+ 2M.$ 

In this case we have 
\begin{equation}\label{ThirdCase}
  \Psi_T(\tau)= \frac{1}{T} \me\big\{ t\in [T, 2T]: \Lv(\sigma+it) \in  \mathcal{V}_J(e^{\tau}, M)\big\},
\end{equation}
 where $\mathcal{V}_J(y, M)$ is the bounded subset of $\mathbb{R}^{2J}$ defined by
\begin{align*}
\mathcal{V}_J(y, M)
:= \Big\{ (u_1, \dots, u_J, v_1, \dots, v_J) \in \mathbb{R}^{2J} : \  & |u_j|, |v_j| <M \text{ for all } 1\le j\le J,\\
& \text{ and } \big|\sum_{j=1}^J a_j e^{u_j+iv_j}\big| > y \Big\}.
\end{align*}
Similarly as before, we cover $\mathcal{V}_J(e^{\tau}, M)$ with $\widetilde{K}(\tau)$ hypercubes $\widetilde{\mathcal{B}}_k(\tau)$
 with non-empty intersection with $\mathcal{V}_J(e^{\tau}, M)$, and with
sides of length $\varepsilon$. The number of such hypercubes is 
$$ \widetilde{K}(\tau) \asymp \frac{\textup{Vol}(\mathcal V_J(e^{\tau}, M))}{\varepsilon^{2J}} \ll \left(\frac{M}{\varepsilon}\right)^{2J}.
$$
Now, let $1\leq k\leq \widetilde{K}(\tau)$ and $(u_1, \dots, u_J, v_1, \dots, v_J) \in \widetilde{\mathcal{B}}_k(\tau) \cap \mathcal{V}_J(e^{\tau}, M)$. Then, for any $(x_1, \dots, x_J, y_1, \dots, y_J)\in \widetilde{\mathcal{B}}_k(\tau)$ we have $|x_j-u_j|\leq \varepsilon$ and 
$|y_j-v_j|\leq \varepsilon$ for all $1\leq j\leq J$. Hence, we deduce that $|x_j|, |y_j|<2M$ for all $1\leq j\leq J$ and 
$$ \big|\sum_{j=1}^J a_j e^{x_j+iy_j}\big|=\big|\sum_{j=1}^J a_j e^{u_j+iv_j}\big| +O(\varepsilon e^M)> e^{\tau}- \varepsilon e^{2M}$$
if $T$ is sufficiently large, since $e^{\tau}> 2\varepsilon e^{2M}$ by our assumption. Therefore, we have shown that
\begin{equation}\label{UpperLowerBoxes2}
\mathcal{V}_J(e^{\tau}, M) \subset \left(\bigcup_{k \le \widetilde{K}(\tau)} \widetilde{\mathcal{B}}_k(\tau) \right) \subset \mathcal{V}_J(e^{\tau}-\varepsilon e^{2M}, 2M).
\end{equation}
Thus, it follows from Theorem \ref{Discrepancy} that
\begin{equation}\label{UpperBThirdCase}
\begin{aligned}
\Psi_T(\tau)& \leq 
\sum_{k=1}^{\widetilde{K}(\tau)} \frac{1}{T} \me\big\{ t\in [T, 2T]: \Lv(\sigma+it) \in \widetilde{\mathcal{B}}_k(\tau)\big\}\\
&= \sum_{k=1}^{\widetilde{K}(\tau)}\pr\left(\Lv(\sigma, \X)\in \widetilde{\mathcal{B}}_k(\tau)\right) + O\left(\frac{\widetilde{K}(\tau)}{(\log T)^{\sigma}}\right)\\
&\leq \pr\left(\Lv(\sigma, \X)\in \mathcal{V}_J(e^{\tau}-\varepsilon e^{2M}, 2M)\right) + O\left(\frac{(\log\log T)^J}{\varepsilon^{2J}(\log T)^{\sigma}}\right)\\
& = \Psr(\tau)+ O\left(\pr\left(\Lv(\sigma, \X) \notin (-M, M)^{2J}\right) +\varepsilon e^{2M}
+ \frac{(\log\log T)^J}{\varepsilon^{2J}(\log T)^{\sigma}}\right),
\end{aligned}
\end{equation}
where in the last estimate we have used that
$$ \pr\left(e^{\tau}-\varepsilon e^{2M}< |E(\sigma, \X)|\leq e^{\tau}\right) \ll \varepsilon e^{2M}$$
since $E(\sigma, \X)$ is an absolutely continuous random variable. Now, by Lemma \ref{MomentsLogR} we have 
\begin{align*} \pr\left(\Lv(\sigma, \X) \notin (-M, M)^{2J}\right) &\leq \sum_{j=1}^J \pr(\log |L_j(\sigma, \X)|\geq M) \\
& \leq \sum_{j=1}^{J} \frac{\ex\left( |\log L_j(\sigma, \X)|^{2k}\right)}{M^{2k}} \ll  \left(\frac{C_4 k}{M^2}\right)^{k} \ll \frac{1}{(\log T)^{2A}},
\end{align*}
if $A$ is suitably large. Thus, we have shown that 
\begin{equation}\label{UB3}\Psi_T(\tau) \leq \Psr(\tau)+ O\left(\varepsilon e^{2M}
+ \frac{(\log\log T)^J}{\varepsilon^{2J}(\log T)^{\sigma}}\right).
\end{equation}

We now proceed to prove the corresponding lower bound. Let $\tau_1$ be such that $e^{\tau}=e^{\tau_1}-\varepsilon e^{2M}$. Then, it follows from \eqref{UpperLowerBoxes2} and Theorem \ref{Discrepancy} that
\begin{equation}\label{LowerBThirdCase}
\begin{aligned}
& \frac{1}{T} \me\big\{ t\in [T, 2T]: \Lv(\sigma+it) \in \mathcal{V}_J(e^{\tau}, 2M)\big\} \\
& \geq 
\sum_{k=1}^{\widetilde{K}(\tau_1)} \frac{1}{T} \me\big\{ t\in [T, 2T]: \Lv(\sigma+it) \in \widetilde{\mathcal{B}}_k(\tau_1)\big\}\\
&= \sum_{k=1}^{\widetilde{K}(\tau_1)}\pr\left(\Lv(\sigma, \X)\in \widetilde{\mathcal{B}}_k(\tau_1)\right) + O\left(\frac{\widetilde{K}(\tau_1)}{(\log T)^{\sigma}}\right)\\
&\geq \Psr(\tau_1) + O\left(\frac{(\log\log T)^J}{\varepsilon^{2J}(\log T)^{\sigma}}\right).\\
\end{aligned}
\end{equation}
Moreover, by \eqref{MeasureS1} we have 
$$ \frac{1}{T} \me\big\{ t\in [T, 2T]: \Lv(\sigma+it) \in \mathcal{V}_J(e^{\tau}, 2M)\big\}= \Psi_T(\tau) + O\left(\frac{T}{(\log T)^{2A}}\right).$$
Finally, we use that $E(\sigma, \X)$ is an absolutely continuous random variable to deduce that 
$$ \Psr(\tau_1)= \Psr(\tau) + O\left(\pr\left(e^{\tau}< |E(\sigma, \X)|\leq e^{\tau}+\varepsilon e^{2M}\right)\right)= \Psr(\tau)+O\left(\varepsilon e^{2M}\right).$$
Inserting these estimates in \eqref{LowerBThirdCase} yields
\begin{equation}\label{LB3}
\Psi_T(\tau) \geq \Psr(\tau)+ O\left(\varepsilon e^{2M}
+ \frac{(\log\log T)^J}{\varepsilon^{2J}(\log T)^{\sigma}}\right).
\end{equation}
The desired result follows by combining \eqref{UB3} and \eqref{LB3} and choosing $\varepsilon= (\log T)^{-\sigma/(2J+1)}$.

\end{proof}

\begin{proof}[Proof of Theorem \ref{AsympLogE}]
We consider the following integral 
$$
 \int_{-M^4}^{M^4}\Psi_T(\tau)d\tau = \int_{-M^4}^{M^4} \frac{1}{T} \int_{\substack{ t\in S(T)\\ \log |E(\sigma+it, Q)|>\tau}} dt= \frac{1}{T} \int_{t\in S(T)} (\log|E(\sigma+it, Q)|+M^4)dt,$$
 where the last equality follows since for all $t\in S(T)$ we have $\max_j|L_j(\sigma+it)|\leq M$ and hence $\log|E(\sigma+it, Q)| = \log| \sum_{j=1}^J a_j L_j(\sigma+it)|\leq M^4$  if $T$ is suitably large. Combining this identity with \eqref{TrunIntLog} and using that $\me(S(T))=T\Psi_T(-M^4)$ we obtain
\begin{equation}\label{MainApprox}
\frac{1}{T}\int_T^{2T}  \log |E(\sigma+it, Q)|dt =  \int_{-M^4}^{M^4}\Psi_T(\tau)d\tau -M^4 \Psi_T(-M^4) + O\left(\frac{1}{(\log T)^A}\right).
\end{equation}

We now repeat the exact same argument but with the random model $E(\sigma, \X)$ instead of the Epstein zeta function. 
Using the same argument leading to \eqref{TrunIntLog} but with Lemma \ref{MomentsLogR} instead of Lemmas \ref{MomentsLogE} and \ref{MomentsOneL}, we deduce similarly that 
$$\ex\left(\log |E(\sigma, \X)|\right)= \ex\left(\mathbf{1}_{\mathcal{S}} \cdot \log |E(\sigma, \X)|\right)+ O\left(\frac{1}{(\log T)^A}\right),$$
where $\mathbf{1}_{\mathcal{S}}$ is the indicator function of $\mathcal{S}$. 
Therefore, reproducing the argument leading to \eqref{MainApprox} we obtain
\begin{equation}\label{MainApprox2}
\ex\left(\log |E(\sigma, \X)|\right)=  \int_{-M^4}^{M^4}\Psr(\tau)d\tau -M^4 \Psr(-M^4) + O\left(\frac{1}{(\log T)^A}\right).
\end{equation}
Finally, it follows from Proposition \ref{CoveringResult} that
$$ \int_{-M^4}^{M^4}\Psi_T(\tau)d\tau -M^4 \Psi_T(-M^4) - \left(\int_{-M^4}^{M^4}\Psr(\tau)d\tau -M^4 \Psr(-M^4)\right) \ll \frac{M^4e^{2M}}{(\log T)^{\sigma/(2J+1)}}.
$$
Combining this bound with \eqref{MainApprox} and \eqref{MainApprox2} completes the proof.
\end{proof}


\end{document}